\documentclass{amsart}
\usepackage{graphicx}
\usepackage[colorlinks,linkcolor=blue,citecolor=red,linktocpage=true]{hyperref}

\newtheorem{thm}{Theorem}[section]
\newtheorem{proposition}[thm]{Proposition}
\newtheorem{lemma}[thm]{Lemma}
\newtheorem{corollary}[thm]{Corollary}
\theoremstyle{definition}
\newtheorem{definition}[thm]{Definition}
\newtheorem{example}[thm]{Example}

\DeclareMathOperator{\rank}{rank}%
\DeclareMathOperator{\codim}{codim}%
\newcommand{\cunita}{{\scriptstyle\wedge}}

\newcommand{\cuna}{\mathchoice{{\textstyle\wedge}}%
    {{\wedge}}%
    {{\textstyle\wedge}}%
    {{\scriptstyle\wedge}}}

\begin{document}

\markboth{J. Carrillo-Pacheco and F. Zaldivar}{Codes on Linear Sections of Grassmannians}
\title{Codes on Linear Sections of Grassmannians}
\author{Jes\'us Carrillo-Pacheco}
\address{Academia de Matem\'aticas \\
Universidad Aut\'onoma de la Ciudad de M\'exico\\
09790, M\'exico, D. F., M\'exico.}
\email{jesus.carrillo@uacm.edu.mx}

\author{ Felipe Zaldivar}
\address{Departamento de Matem\'aticas \\
Universidad Aut\'onoma Metropolitana-I\\
09340, M\'exico, D. F., M\'exico}
 \email{fz@xanum.uam.mx}


\begin{abstract}
We study algebraic geometry linear codes defined by linear sections of the Grassmannian variety as  codes associated to FFN$(1,q)$-projective varieties. As a consequence, we show that
Schubert, Lagran\-gian-Grassmannian, and isotropic Grassmannian codes
are special instances of codes defined by linear sections of the Grassmannian variety.
\end{abstract}



\maketitle

\section{Introduction}\label{sec1}

Let ${\mathbb F}_q$ be a finite field with $q$ elements, and denote
by $\overline{\mathbb F}_q$ an algebraic closure of ${\mathbb F}_q$. For $E$ a vector space over ${\mathbb F}_q$ of finite dimensiom $k$, let $\overline{E}=E\otimes_{{\mathbb F}_q}\overline{\mathbb F}_q$ the corresponding vector space over the algebraically closed field $\overline{\mathbb F}_q$. We will be considering algebraic varieties in the projective space ${\mathbb P}(\overline{E})={\mathbb P}^{k-1}(\overline{\mathbb
F}_q)$. Recall now that a projective variety $X\subseteq{\mathbb P}^{k-1}(\overline{\mathbb
F}_q)$ is defined over the finite field ${\mathbb F}_q$ if its
vanishing ideal can be generated by polynomials with coefficients in
${\mathbb F}_q$. Also, a projective variety $X\subseteq {\mathbb
P}^{k-1}(\overline{\mathbb F}_q)$ is \textit{non-degenerate} if $X$ is
not contained in a hyperplane of ${\mathbb P}^{k-1}(\overline{\mathbb
F}_q)$. For a projective variety $X\subseteq {\mathbb
P}^{k-1}(\overline{\mathbb F}_q)$ defined over ${\mathbb F}_q$, we denote by $X({\mathbb F}_q)$ its set of ${\mathbb F}_q$-rational points.

The arithmetic counterpart of these geometric concepts is the notion of \textit{non-degenerate projective system}, that is, set of points
$\chi\subseteq {\mathbb P}^{k-1}({\mathbb F}_q)$ not contained in
a hyperplane of ${\mathbb P}^{k-1}({\mathbb F}_q)$. The question of when a non-degenerate projective variety $X\subseteq {\mathbb P}^{k-1}(\overline{\mathbb F}_q)$ descends to a non-degenerate projective system $X({\mathbb F}_q)\subseteq{\mathbb P}^{k-1}({\mathbb F}_q)$ is captured by the so-called FFN$(1,q)$-property \cite{0}, \cite{3.1}, that is,  projective varieties $X\subseteq {\mathbb P}^{k-1}(\overline{\mathbb F}_q)$ that satisfy that
every homogeneous linear polynomial  with coefficients in ${\mathbb F}_q$ that vanishes on its set of
${\mathbb F}_q$-rational points $X({\mathbb F}_q)$, also vanishes on
the whole $X\subseteq{\mathbb P}^{k-1}(\overline{\mathbb F}_q)$. These varieties are important in  coding theory, since
by \cite{23} and \cite{30}, to their sets of ${\mathbb F}_q$-rational points $X({\mathbb F}_q)\subseteq{\mathbb P}^{k-1}({\mathbb F}_q)$  there is associated  a non-degenerate $[n,k]_q$-linear code $C_{X({\mathbb F}_q)}\subseteq{\mathbb F}_q^n$ of  length $n=|X({\mathbb F}_q)|$, dimension $k$, and minimum distance 
$$d=d(C_{X({\mathbb F}_q)})=n-\max\{\left|X({\mathbb F}_q)\cap H\right|\; :\; H\; \text{is a hyperplane of ${\mathbb P}^{k-1}({\mathbb F}_q)$}\}.$$
Moreover, the higher weights $d_r=d_r(C_{X({\mathbb F}_q)})$ of $C_{X({\mathbb F}_q)}$ are given by
$$d_r=n-\max\{|X({\mathbb F}_q)\cap H| : H\subseteq{\mathbb P}^{k-1}({\mathbb F}_q)\; \text{a codimension $r$ projective subspace}\}.$$

There are several families of projective algebraic varieties $X$, defined over a finite field ${\mathbb F}_q$ such that the set of ${\mathbb F}_q$-rational points $X({\mathbb F}_q)\subseteq {\mathbb P}^{k-1}({\mathbb F}_q)$ is a non-degenerate system, \cite{3.1}. Here we will be interested on Grassmann varietes and some of their subvarietes. Recall that if $E$ is a vector space of dimension $m$, defined over ${\mathbb F}_q$,  and $\overline{E}=E\otimes_{{\mathbb F}_q}\overline{\mathbb F}_q$,  the Grassmann variety $G(\ell,\overline{E})=G(\ell,m)$ is the set of all vector subspaces of $\overline{E}$ of dimension $\ell$. The Pl\"ucker embedding of $G(\ell,m)$ into the projective space ${\mathbb P}(\cuna^{\ell}\overline{E})={\mathbb P}^{k-1}(\overline{\mathbb F}_q)$, for $k=\binom{m}{\ell}$ is non-degenerate. 
Moreover, the set of ${\mathbb F}_q$-rational points $G(\ell,m)({\mathbb F}_q)$ of the Grassmannian is a non-degenerate projective system in ${\mathbb P}^{k-1}({\mathbb F}_q)$, see \cite{0.1}. Hence, it defines an $[n,k,d]_q$-linear code, where
$n=\big|G(\ell,m)({\mathbb F}_q)\big|=\left[\begin{smallmatrix}
m\\
\ell
\end{smallmatrix}\right]_q$ (the Gaussian $q$-binomial coefficient), 
$k=\binom{m}{\ell}$ and $d=q^{\delta}$, for $\delta=\ell(m-\ell)$, see \cite{19} and \cite{20}.

The main contribution of this paper is to give a uniform construction of  codes associated to linear sections of the Grassmann variety as codes given by FFN$(1,q)$-projective varieties. As a by-product, we show that the
Schubert codes of Ghorpade and Lachaud \cite{5}, Lagrangian-Grassmannian codes of the authors  \cite{3}, and the isotropic Grassmannian codes of  Cardinali and Giuzzi in \cite{2.1}, are instances of codes associated to linear sections of the Grassmannian.

The paper is organized as follows:  In Section \ref{sec2} we establish some general facts on linear sections of Grassmannians, where the main results are Propositions \ref{prop2.1},  \ref{prop2.2}, and  \ref{prop2.3}. Section \ref{sec3} reinterprets and introduces examples of algebraic-geometry codes as codes associated to linear sections of Grassmannians   in the light of the results of Section \ref{sec2}. In Section \ref{sec4} we obtain general bounds for the higher weights of these codes.

\section{Preliminaries and Linear Sections of Grassmannians}\label{sec2}

Let $X\subseteq {\mathbb P}(\overline{E})$ be an irreducible projective variety, where $E$ is a vector space of finite dimension over  a finite  field ${\mathbb F}_q$. Let $\overline{\mathbb F}_q$ be an algebraic closure of ${\mathbb F}_q$ and $\overline{E}=E\otimes_{{\mathbb F}_q}\overline{\mathbb F}_q$. Let $\overline{K}=\{h\in \overline{E}^*: h(x)=0\; \text{for all} \; x \in X \}$, where $\overline{E}^*$ is the dual space of $\overline{E}$. Let $L(X)=\bigcap_{h\in K}\ker h$ be the linear hull of $X$, that is, the smallest linear subspace of ${\mathbb P}(\overline{E})$ that contains $X$. Thus, if
$\overline{V}=\{x\in \overline{E} :h(x)=0\;\text{for all}\; h\in K\}$, 
then, $L(X)={\mathbb P}(\overline{V})$. For the ring extension ${\mathbb F}_q[x_0,\ldots,x_N]\subseteq \overline{\mathbb F}_q[x_0,\ldots, x_N]$, if $I\subseteq {\mathbb F}[x_0,\ldots,x_N]$ is an ideal, we denote by $I_{\overline{\mathbb F}_q}$ its extension to $\overline{\mathbb F}_q[x_0,\ldots,x_N]$.
\begin{lemma}\label{lemma1} 
Let $X\subseteq {\mathbb P}(\overline{E})$ be an irreducible projective variety. 
\medskip

\noindent\emph{(1)} The embedding $X \hookrightarrow L(X)$ is non degenerate.
\medskip

\noindent\emph{(2)} Let $I_{\overline{\mathbb F}_q}(X)=\langle
f_1,\ldots , f_M , g_1 ,\ldots ,g_N \rangle$ be the vanishing ideal of $X$, where $f_i$
and $g_j$ are homogeneous, $\deg f_i\geq 2$ and $\deg g_j =1$. Then,
$L(X)=Z_{\overline{\mathbb F}_q}\langle g_1, \ldots ,g_N \rangle$.
\end{lemma}
\begin{proof}
We just need to prove the second part.
Clearly, $L(X)\subseteq Z_{\overline{\mathbb F}_q}\langle g_1, \ldots ,g_N
\rangle$. Now, if  $h\in \overline{K}$, then $h \in I_{\overline{\mathbb F}_q}(X)$ and
$h=\sum_{i=1}^{M}\alpha_{i} f_i + \sum_{j=1}^{N}\beta_j g_j$, with
$\alpha_{i}$ and $\beta_{j}$ polynomials with coefficients in
$\overline{\mathbb F}_q$.  Cleary $h=\sum_{j=1}^{N}\beta_{j}^{'} g_j$, where
$\beta_{j}^{'} \in \overline{\mathbb F}_q$. Thus, if  $x \in Z_{\overline{\mathbb F}_q}\langle
g_1, \ldots ,g_N \rangle$, then $h(x)=\sum_{j=1}^{N}\beta_{j}^{'}
g_j(x)=0$ and hence $Z_{\overline{\mathbb F}_q}\langle g_1, \ldots ,g_N \rangle
\subseteq L(X)$.
\end{proof}

For the finite dimensional ${\mathbb F}_q$-vector space $E$,  let $\chi=\{P_1,\ldots,P_r\}$ be a finite subset of points of the projective space ${\mathbb P}(E)$.
 Let $K=\{h\in E^*\; :\; h(P_1)=\cdots=h(P_r)=0\}$, where $E^*$ is the dual space of $E$ and the linear forms  $h\in E^*$ have coefficients in ${\mathbb F}_q$. Let  $W=\{x\in E\; :\; h(x)=0\;\;\text{for all}\;\; h\in K\}$.
Notice that in the last part of the proof of Lemma \ref{lemma1} we must have that
${\mathbb P}(W)=Z_{{\mathbb F}_q}\langle h_1,\ldots,h_s\rangle$,  if $h_1,\ldots,h_s$ is a basis of $K$. We quote the following lemma and its corollary from \cite{3.1}:
\begin{lemma}\emph{(\cite[Lemma 2.1]{3.1}).}\label{lemma2}
With the notation above, ${\mathbb P}(W)$ is the smallest linear
subvariety of ${\mathbb P}(E)$ that contains $\chi$.
\end{lemma}
\begin{corollary}\emph{(\cite[Corollary 1]{3.1}).}\label{corollary4}
With the notation above, $\chi$ is a non-degenerate projective
system in ${\mathbb P}(W)$.
\end{corollary}
The following consequence is immediate (see also \cite[Corollay 3]{3.1}):
\begin{corollary}\label{corollary5}
With the notation above, let  $B$ be the matrix of the  system of linear equations  $h_1=0,\ldots,h_s=0$, and $C_{\chi}$  the linear code associated to the nondegenerate projective system $\chi$. Then,
$C_{\chi}$ is an $[n,k]_q$-linear code, where $n=|\chi|$ and $k=s-\rank B$.
\end{corollary}
Write ${\mathbb P}(\overline{\mathbb E})={\mathbb P}^n(\overline{\mathbb F}_q)$ with homogeneous coordinates $x_0,\ldots, x_n$. For an {FFN}$(1,q)$-projective variety  $X\subseteq {\mathbb P}(\overline{\mathbb E})$   defined over the finite field ${\mathbb F}_q$ write its vanishing ideal as 
$I_{\overline{\mathbb F}_q}(X)=\langle f_1,\ldots ,f_M,g_1,\ldots ,g_N \rangle_{\overline{\mathbb F}_q}$, with $f_i,g_j\in{\mathbb F}_q[x_0,\ldots,x_n]$ homogeneous forms with $\deg f_i\geq 2$ and $\deg g_j = 1$. We keep this notation for the rest of this section.
\begin{lemma}\label{lem2.3}
Let $X\subseteq {\mathbb P}(\overline{\mathbb E})$ be an \emph{FFN}$(1,q)$ projective
variety defined over the finite field ${\mathbb F}_q$.  Then, $X({\mathbb F}_q)$ is a
non-degenerated projective system in ${\mathbb P}(V)$, where
$V=\bigcap \ker(g_i)$, over ${\mathbb F}_q$. 
\end{lemma}

\begin{proof}
If $H\subseteq {\mathbb P}(V)$ is hyperplane, say $H=Z_{{\mathbb F}_q}(h)$ for $h$ i a homogeneous linear  with coefficients in ${\mathbb F}_q$ and if $X({\mathbb F}_q)\subseteq H$, that is  $h(X({\mathbb F}_q))=0$, since $X$ satisfies the FFN$(1,q)$ property, then $h(X)=0$, and hence $h \in I_{{\overline{\mathbb F}_q}}(X)$ and since it is linear, $h=\sum_{j=1}^{N}b_jg_j$, with $b_j\in\overline{\mathbb F}_q$. Therefore, $\overline{H}:=Z_{\overline{\mathbb F}_q}(h)={\mathbb P}(\overline{V})$, and thus $H=\overline{H}\cap {\mathbb P}(V)={\mathbb P}(V)$.
\end{proof}
\begin{proposition}\label{prop2.1}
Let   $X$ be a projective variety defined over a finite field  ${\mathbb F}_q$, and let $J:= \langle f_i,g_j,x_k^q-x_k:0\leq k\leq n\rangle_{\overline{\mathbb F}_q}$. If $h$ is a homogeneous linear form with coefficients in ${\mathbb F}_q$
such that $h^q-h \in J$ and $h \in \sqrt{J}$, then, $h^q \in J$.
\end{proposition}

\begin{proof}
Since $h \in \sqrt{ J}$, there exists an $m\in {\mathbb N}$ such that $h^m \in  J$.
We distinguish two cases. Firstly, if $m\leq q$, then $h^q=h^{m}h^{q-m} \in J$.
 Secondly, if  $m>q$, write $m=nq+r$ with
$n,r \in {\mathbb N}$ and $r<q$. We do induction on $n$: 
If $n=1$, then $m=q+r$ and $h^m-h^{r+1}=h^{q+r}-h^{r+1}=h^r(h^q-h) \in J$, since  $h^q-h\in J$. Now, since $h^m\in J$, from $h^m-h^{r+1}\in J$ it follows that $h^{r+1}\in J$, with $r<q$, that is $r+1\leq q$, and by the first case it follows that $h^q\in J$. Assume that the result holds up to $n$, i.e.,
if $h^m \in J$ and  $m=sq+r$, with $1\leq s \leq n$ and
$r<q$, then $h^q \in J$. 

Now, if  $m=(n+1)q+r$, with $r<q$, the hypothesis $h^q-h\in J$ implies that $h^m-h^{m-q+1}=h^{m-q}(h^q-h)\in J$, and since $h^m\in J$, it follows that $h^{m-q+1}\in J$. Thus, $h^{nq+(r+1)}=h^{m-q+1}\in J$. Hence, if $r+1<q$, then $h^q\in J$ 
by induction hypothesis.
Now, if $r+1=q$, since $h^{(n+1)q}=h^{nq+q}=h^{nq+r+1}\in J$, and on the other hand, since $h^q-h\in J$, then
$h^{(n+1)q}-h^{nq+1}=h^{nq}(h^q-h)\in J$, it follows that
$h^{nq+1}\in J$, and by induction hypothesis  $h^{q}\in J$.  
\end{proof}
\begin{proposition}\label{prop2.2}
Let   $X$ be a projective variety defined over a finite field  ${\mathbb
F}_q$ such that $X({\mathbb F}_q)\neq \emptyset$.  If $h$ is a linear
homogeneous form with coefficients in ${\mathbb F}_q$ which vanishes in
$X({\mathbb F}_q)$, then there exists a linear  homogeneous form $h'$ with
coefficients in $\overline {\mathbb F}_q$ such that $h-h' \in
\langle g_1,\ldots ,g_N \rangle_{\overline{\mathbb F}_q}$.
\end{proposition}
\begin{proof}
Let $C(X)\subseteq {\mathbb A}^{n+1}_{\overline{\mathbb F}_q}$ be
the affine cone of $X\subseteq {\mathbb P}^n_{\overline{\mathbb
F}_q}$. Thus, its vanishing ideal is $I_{\overline{\mathbb F}_q}(C(X))=I_{\overline{\mathbb
F}_q}(X)= \langle f_i,g_j\rangle\subseteq \overline{\mathbb
F}_q[x_0,\ldots,x_n]$ with $f_i,g_j\in {\mathbb
F}_q[x_0,\ldots,x_n]$ as before. 
Then, the set of ${\mathbb
F}_q$-rational points of $C(X)$ is
\begin{align*}
C(X)({\mathbb F}_q)&=C(X)\cap {\mathbb A}^{n+1}({\mathbb F}_q)=Z_{\overline{\mathbb F}_q}\langle f_i,g_j\rangle\cap Z_{\overline{\mathbb F}_q}\langle x_k^q-x_k: 0\leq k\leq n\rangle\\
&=Z_{\overline{\mathbb F}_q}(J),\qquad\text{for $J=\langle f_i,g_j,x_k^q-x_k:0\leq k\leq n\rangle$}.
\end{align*}
By the Nullstellensatz,
$$
I_{\overline{\mathbb F}_q}(C(X)({\mathbb F}_q))=I_{\overline{\mathbb F}_q}Z_{\overline{\mathbb F}_q}(J)=\sqrt{J_{\overline{\mathbb F}_q}}.
$$

Now, let $h\in I_{{\mathbb F}_q}(C(X)({\mathbb F}_q))$ be a linear
form, say $h=a_0x_0+\cdots+a_nx_n$, with $a_i\in{\mathbb F}_q$.
Then, $h^q=a_0^qx_0^q+\cdots+a_n^qx_n^q=a_0x_0^q+\cdots+a_nx_n^q$ and thus
$$h^q-h=\sum_{k=0}^na_k(x_k^q-x_k)\in\langle f_i,g_j,x_k^q-x_k:0\leq k\leq n\rangle_{\overline{\mathbb F}_q}.\leqno(1)$$
On the other hand, since $h\in  I_{\overline{\mathbb
F}_q}(C(X)({\mathbb F}_q))=\sqrt{J_{\overline{\mathbb F}_q}}$, there exists an 
$m>0$ such that $h^m\in J_{\overline{\mathbb F}_q}$. By Proposition \ref{prop2.1}, $h^{p} \in J_{\overline{\mathbb F}_q}$. Writting $h$ as a linear combination of the polynomials $f_i, g_j, x_k^q-x$, with coefficients  $\alpha_i,\beta_j,\gamma_k
 \in \overline{\mathbb F}_q$, we obtain 
$$h=\Big(\sum_i\alpha_if_i+\sum_k\gamma_k
x_k^q\Big)+\Big(\sum_j\beta_jg_j+\sum_k\gamma_k x_k\Big),$$
and since $h$ is linear and homogeneous, it follows that
$(\sum_i\alpha_if_i+\sum_k\gamma_k x_k^q)=0$.  Thus,
$h=\sum_j\beta_jg_j+\sum_k\gamma_k x_k$ where $\beta_j$ and
$\gamma_k \in \overline{\mathbb F}_q$. Hence, if $h':=\sum_k\gamma_k x_k$,
where $\gamma_k$ are as before,  it follows that $h-h'\in\langle g_1,\ldots , g_N\rangle_{\overline{\mathbb F}_q}$.
\end{proof}
\begin{corollary}\label{corollary7}
Let   $X\subseteq {\mathbb P}^n(\overline{\mathbb F}_q)$  be a projective variety defined over a finite field  ${\mathbb
F}_q$ and let $H$ be a hyperplane in ${\mathbb P}^{n}({\mathbb F}_q)$ such that
$X({\mathbb F}_q)\subseteq H$. Then, there exists a  hyperplane $H'$ in
${\mathbb P}^{n}(\overline{\mathbb F}_q)$ such that  $X({\mathbb
F}_q)\subseteq H'$ but $X$ is not contained in $H'$, and moreover
$H=\bigcap_{j=1}^{N}H_j \bigcap H'$.
\end{corollary}
\begin{corollary}\label{corollary8}
Let   $X\subseteq {\mathbb P}^n(\overline{\mathbb F}_q)$ be a projective variety defined over the finite field  ${\mathbb F}_q$. Then, $X$ satifies the
\emph{FFN}$(1,q)$-property if and only if for
each $h\in I_{{\mathbb F}_q}(X({\mathbb F}_q))$ we have that $h'\in\langle g_1,\ldots , g_N\rangle_{\overline{\mathbb F}_q}$.
\end{corollary}
\begin{proof}
Clearly, $X$ satifies the FFN$(1,q)$-property if and only if for each linear homogeneous form
$h\in I_{{\mathbb F}_q}(X({\mathbb F}_q))$ we have  that $h(X)=0$, and this is equivalent to
$h\in\langle g_1,\ldots, g_N\rangle_{\overline{\mathbb
F}_q}$, which, by Proposition \ref{prop2.2}, it happens  if and only if $h'\in\langle g_1,\ldots ,
g_N\rangle_{\overline{\mathbb F}_q}$.
\end{proof}
\noindent{\bf Linear Sections of Grassmannians.} 
 For an ideal $J\subseteq {\mathbb F}_q[x_0,\ldots,x_m]$ denote by
$Z_{{\mathbb F}_q}(J)\subseteq{\mathbb P}^m({\mathbb F}_q)$ and 
$Z_{\overline{\mathbb F}_q}(J)\subseteq{\mathbb
P}^n(\overline{\mathbb F}_q)$ the zero sets of $J$ in ${\mathbb
P}^m({\mathbb F}_q)$ and ${\mathbb P}^m(\overline{\mathbb F}_q)$,
respectively. In particular, for a linear form $h\in {\mathbb F}_q[x_0,\ldots,x_m]$, 
 $H=Z_{{\mathbb F}_q}(h)\subseteq{\mathbb P}^m({\mathbb F}_q)$
and $\overline{H}=Z_{\overline{\mathbb F}_q}(h)\subseteq{\mathbb
P}^m(\overline{\mathbb F}_q)$ denote the corresponding hyperplanes. For
an ideal $I\subseteq {\mathbb F}_q[x_0,\ldots,x_m]$ we will  denote
by $I_{\overline{\mathbb F}_q}\subseteq \overline{\mathbb
F}_q[x_0,\ldots,x_m]$ the corresponding ideal in $\overline{\mathbb
F}_q[x_0,\ldots,x_m]$.

Let $E$ be a vector space of finite dimension $m$ over a finite field ${\mathbb F}_q$ and let
$G(\ell,m)=G(\ell,\overline{E})\subseteq {\mathbb P}(\wedge^{\ell}\overline{E})$ be the Grassmannian variety embedded, via the Pl\"ucker map, in the projective space ${\mathbb P}^{k-1}(\overline{\mathbb F}_q)={\mathbb P}(\wedge^{\ell}\overline{E})$, where $k=\binom{m}{\ell}$. If $X\subseteq{\mathbb P}^{k-1}(\overline{\mathbb F}_q)$ is an
irreducible projective variety defined over the finite field ${\mathbb F}_q$, we say that $X$ is a \emph{linear  section of the Grassmannian variety} $G(\ell,m)$ if
$X=G(\ell,m)\cap{\mathbb P}(V)$, where $V\subseteq \wedge^{\ell}\overline{E}$ is a vector subspace.
 Using the Pl\"ucker coordinates $p_{\alpha}$ for the Grassmannian $G(\ell,m)\subseteq {\mathbb P}^{k-1}(\overline{\mathbb F}_q)$, where the indexes $\alpha$ run in $I(\ell, m)=\{\alpha=(\alpha_1,\ldots,\alpha_{\ell})\in{\mathbb Z}^{\ell}:1\leq\alpha_1<\cdots<\alpha_{\ell}\leq m\}$,  consider also the set $I[\ell,m]$ of non-ordered $\ell$-tuples of the set $[m]=\{1,\ldots,m\}$. The set $I(\ell,m)$ is given the Bruhat order, and for an ordered $\ell$-tuple $\alpha=(\alpha_1,\ldots,\alpha_{\ell})\in I(\ell,m)$ its support is the set $\text{supp}(\alpha)=\{\alpha_1,\ldots,\alpha_{\ell}\}\in I[\ell,m]$. 
Following \cite{6},  
for  $\Lambda \subseteq I[\ell, m]$ we define the linear section
$$E_{\Lambda}=\{p=(p_{\alpha})\in G(\ell,m): p_{\alpha}=0 \;\text{for all $\alpha \in \Lambda$} \}=G(\ell,m)\cap {\mathbb P}(V)$$
where $V=Z\langle x_\alpha : \alpha \in \Lambda \rangle$. Let
$Q=I(G(\ell,m))$ denote the vanishing ideal de $G(\ell,m)$.
In \cite{6} it is shown that if $\Lambda \subseteq I[\ell, m]$
is a close subset, then the vanishing ideal of
$E_{\Lambda}$,  $I(E_{\Lambda})= Q + \langle x_\alpha : \alpha \in
\Lambda \rangle$, the ideal generated by $Q$ and the indeterminates
corresponding to $\Lambda$, is a radical ideal. 
\begin{proposition}\label{prop2.3}
$E_{\Lambda}$ satisfies the \emph{FFN}$(1,q)$-property
\end{proposition}
\begin{proof}
Let $E_{\Lambda}({\mathbb F}_q)$ be the set of ${\mathbb F}_q$-rational points of $E_{\lambda}$ and assume that the linear form $h=\sum_{\alpha \in I[\ell, m]}a_\alpha X_\alpha$   vanishes on $E_{\Lambda}({\mathbb F}_q)$. By definition of $\Lambda$, if $\alpha\in\Lambda$ then $p_{\alpha}=0$, and if  $\lambda \in I(\ell, m)-\Lambda$ (that is, $\lambda\in I(\ell,m)$ but $\text{supp}(\lambda)\not\in \Lambda$), let  $\overline{p}_\lambda=(p_\alpha)_{\alpha \in I[\ell, m]}$ in
$G(\ell, m)$ such that 
$$p_\alpha =\begin{cases}
  1 & \text{if $\text{supp}(\alpha)=\lambda$}, \\
 0& \text{otherwise}.
\end{cases}
 $$
Thus, $h(\overline{p}_\lambda)=a_\lambda=0$ and hence $h=\sum_{\alpha \in \Lambda}a_\alpha
 X_\alpha$, that by definition of $\Lambda$,  vanishes on $E_{\Lambda}$.
\end{proof}
\begin{definition}\label{def2.1}
The ${\mathbb F}_q$-linear code associated to the non-degenerate projective system $E_{\Lambda}({\mathbb F}_q)$  of Proposition \ref{prop2.3} is denoted by
$C_{E_{\Lambda}}$ and will be called a {\it linear section code of the Grassmannian associated to the close set $\Lambda$}. Its parameters are:
\begin{enumerate}
\item[(1)] $n=\left[\begin{smallmatrix}
m\\
\ell
\end{smallmatrix}\right]_q -q^{\delta}-q^{\delta-1}-\cdots -q^{\delta-r+1}$, where $\delta=\ell(m-\ell)$ and $r=|\Lambda|\geq 2$,
\item[(2)] $k=\binom{m}{\ell}-\rank B $, where $B$ is the matrix
associated  to the homogeneous system of linear equations $\{ x_\alpha=0 :
\alpha \in \Lambda \}$,
\item[(3)] $\dim E_{\Lambda}=\delta-2$.
\end{enumerate}
\end{definition}

\section{Examples of Linear Sections  and Schubert calculus}\label{sec3}
\begin{example}[Schubert codes]\label{example3.1}  
If $\lambda\in I(\ell,m)$ and $\Lambda=\{\lambda\}$, then $E_{\lambda}$ is isomorphic to $\Omega_{\lambda}$, the
Schubert variety for a fixed flag and $\lambda \in I(\ell,m)$. It is easy to see
that
$$\Omega_{\lambda}=G(\ell,m)\cap Z(x_{\beta};\beta\not\leq\lambda)$$
and  hence the lenght of the associated code
$C_{\Omega_{\lambda}}$ is $\big|\Omega_{\lambda}({\mathbb F}_q)\big|$ and its dimension is
$\dim C_{\Omega_{\alpha}}=\binom{m}{\ell}-\rank B$, where $B$ is the
associated matrix to the system of linear equations
$Z\{x_{\beta}:\beta\not\leq \lambda\}$.
\end{example}
\begin{example}[Schubert unions linear codes]\label{example3.2}  
Let $\Omega_{\alpha_1},\ldots,\Omega_{\alpha_s}$ be Schubert varieties in the Grassmannian $G(\ell,m)$, and let
\begin{align*}
S_U&=\bigcup_{i=1}^s\Omega_{\alpha_i}=\bigcup_{i=1}^s\big(G(\ell,m)\cap Z_{\overline{\mathbb F}_q}(x_{\beta}:\beta\in I(\ell,m), \; \beta\not\leq\alpha_i)\big)\\
&=G(\ell,m)\cap \bigcup_{i=1}^sZ_{\overline{\mathbb F}_q}\big(x_{\beta}:\beta\in I(\ell,m), \; \beta\not\leq\alpha_i\big).
\end{align*}
Hence,
$$S_U({\mathbb F}_q)=G(\ell,m)({\mathbb F}_q)\cap \bigcup_{i=1}^sZ_{{\mathbb F}_q}\big(x_{\beta}:\beta\in I(\ell,m), \; \beta\not\leq\alpha_i\big).$$
In \cite{11}, it is proved that:
\begin{enumerate}
\item The linear hull of $S_U$ is $L(S_U)=\bigcup_{i=1}^sZ_{\overline{\mathbb F}_q}\big(x_{\beta}:\beta\in I(\ell,m), \; \beta\not\leq\alpha_i\big)$.
\item If $G_U=\bigcup_{i=1}^s\big\{x_{\beta}:\beta\in I(\ell,m), \; \beta\not\leq\alpha\big\}$, then $\dim L(U)=|G_U|$.
\item The number of ${\mathbb F}_q$-rational points in $S_U$ is
$$|S_U({\mathbb F}_q)|=\sum_{(\lambda_1,\ldots,\lambda_{\ell})}q^{\lambda_1+\cdots+\lambda_{\ell}-\frac{\ell(\ell+1)}{2}}.$$
\end{enumerate}
From (1) it follows that $S_U({\mathbb F}_q)$ is a nondegenerate projective system in $L(S_U)({\mathbb F}_q)$. The corresponding non-degenerate ${\mathbb F}_q$-linear code, denoted by $C_{S_U({\mathbb F}_q)}$, has parameters
$n=|S_U({\mathbb F}_q)|$ and dimension $k=|G_U|$. By \cite[Corollary 4]{3.1},  its minimun distance satisfies $d\leq q^{\dim S_U}$.
\end{example}
\begin{example}[Lagrangian-Grassmannian codes]\label{example3.3}
Let $E$ be a symplectic vector space over a field $F$, with  non-degenerate skew-symmetric bilinear form $\langle\,,\,\rangle$, and  even dimension $2n$. A vector subspace $W\subseteq E$ is \emph{isotropic} iff $\langle x,y\rangle=0$ for all $x,y\in W$. Hence, the dimension of $W$ is $\leq n$.
The Lagrangian-Grassmannian variety $L(n,2n)$ is the set
$$L(n,2n)=\big\{W\in G(n,2n)\; :\; W\; \mbox{is isotropic}\big\}.$$
Sending a basis $v_1,\ldots,v_n$ of each $W\in G(n,2n)$   to the class of $v_1\cunita\cdots\cunita v_n$ in
${\mathbb P}\big( \cuna^nE\big)$ we obtain the following description
$$L(n,2n)=\{v_1\cunita\cdots\cunita v_n\in G(n,2n)\,:\,\langle v_i,v_j\rangle=0\;\mbox{for all}\, i,j\}.$$
The number of ${\mathbb F}_q$-rational points of $L(n,2n)$, see \cite{3} for example,  is
$$|L(n,2n)({\mathbb F}_q)|=\prod_{i=1}^n(1+q^i).$$
Now, for the contraction map $f:\cuna ^nE\rightarrow \cuna^{n-2}E$, given by
$$v_1\cunita\cdots\cunita v_n\mapsto \sum_{1\leq r<s\leq n}\langle v_r,v_s\rangle v_1\cunita\cdots\cunita\widehat{v}_r\cunita \cdots\cunita\widehat{v}_s\cunita\cdots\cunita v_n$$
where $\widehat{v}$ means that the corresponding term is omitted, 
denote by ${\mathbb P}(\ker f)$ the projectivization of $\ker f$. Under
the Pl\"ucker embedding, ${\mathbb P}(\ker f)$ is a closed irreducible
subset of ${\mathbb P}(\cuna^nE)$ and ${\mathbb
P}(\ker f)=Z\langle g_1,\ldots,g_N\rangle$ is the  zero set of a
family of linear homogeneous polynomials $g_1,\ldots,g_N$,  and we
may assume that the $g_i$ are a minimal set of generators. In
\cite[Section 3]{3} these linear forms where given explicitly. By \cite[Lemma 1]{3}, 
$L(n,2n)=G(n,2n)\cap {\mathbb P}(\ker f)$. 
The set of rational points $L(n,2n)({\mathbb F}_q) $ is a non-degenerate projective
system in ${\mathbb P}(\ker f)$. 
Indeed, we can be more precise about this result, but first we recall from \cite[Section 3]{3} that for $\alpha\in I(n,2n)$ we denote by $\alpha_{rs}\in I(n-2,2n)$ the sequence obtained from $\alpha$ by deleting the indexes corresponding to $r$ and $s$. Then, let 
$\Pi_{\alpha_{rs}}=\sum_{i=1}^n
a_{i,\alpha_{rs},2n-i+1}X_{i,\alpha_{rs},2n-i+1}$, where
$$a_{i,\alpha_{rs},2n-i+1}=\begin{cases}
1 & \text{if $|\text{\rm supp}\{i, \alpha_{rs}, 2n-i+1 \}| =n$}, \\
0 & \text{otherwise},
\end{cases}$$
and $X_{\alpha}\in  F[X_{\alpha}]_{\alpha\in I(n,2n)}$ are the corresponding indeterminates. With this notation,
${\mathbb P}(\ker f)\subseteq {\mathbb P}(\cuna^nE)$ is the zero set of the $\Pi_{\alpha_{rs}}$, for all $\alpha_{rs}\in I(n-2,2n)$ as before. Let $e_1,\ldots, e_{2n}$ be the standard symplectic basis of $E$, that is $\langle e_i,e_{2n-i+1}\rangle=1$ for $1\leq i\leq n$, and zero otherwise. For $\alpha=(\alpha_1,\ldots,\alpha_n)\in I(n,2n)$, the tensors $e_{\alpha}=e_{\alpha_1}\cunita\cdots\cunita e_{\alpha_n}\in\cuna^nE$ form the usual basis of this vector space.
\begin{lemma}
With the above notation, the only homogeneous linear forms $h$ in the dual space $(\cuna^nE)^*$ that vanish on the set of ${\mathbb F}_q$-rational points $L(n,2n)({\mathbb F}_q)$ are linear combinations of the form
$$\Pi_{\alpha_{rs}}=\sum_{i=1}^nX_{i,\alpha_{rs},2n-i+1}$$
\end{lemma}

\begin{proof}
By induction on $n$, assume first that $n=2$. Then, 
$$I(2,4)=\{(1,2),(1,3),(1,4), (2,3),(2,4), (3,4)\}$$
and notice that 
$\{e_{12},e_{13},e_{24},e_{34}\}\subseteq L(2,4)({\mathbb F}_q)$ since they are totally decomposable and isotropic. 
Now, if $h\in (\cuna^2E)^*$ vanishes on $L(2,4)({\mathbb F}_q)$, then $L(2,4)({\mathbb F}_q)\subseteq Z(h,\Pi)$, where $\Pi=X_{14}+X_{23}$. Suppose that $h=A_{12}X_{12}+A_{13}X_{13}+A_{14}X_{14}+A_{23}X_{23}+A_{24}X_{24}+A_{34}X_{34}$. Since $h(L(2,4)({\mathbb F}_q))=0$, then $h=A_{14}X_{14}+A_{23}X_{23}$, and since $X_{14}+X_{23}=0$, it follows that $h=(A_{14}-A_{23})X_{14}$. By \cite{3.2}, $w=(1,0,1,0,1)\in L(2,4)({\mathbb F}_q)$ and thus $h(w)=(A_{14}-A_{23})1=0$, that is $A_{14}=A_{23}=:A$, and consequently $h=A(X_{14}+X_{23})=A\Pi$, as required. Our induction hypothesis is: For all $k<n$, every $h\in (\cuna^kE)^*$ such that $L(k,2k)({\mathbb F}_q)\subseteq Z\big(h,\Pi_{\alpha_{rs}}:\alpha_{rs}\in I(k-2,2k)\big)$, must be of the form $h=\sum A_{\alpha{rs}}\Pi_{\alpha_{rs}}$, for $\alpha_{rs}\in I(k-2,2k)$.
Assume now that $h\in (\cuna^nE)^*$ and $L(n,2n)({\mathbb F}_q)\subseteq Z\big(h,\Pi_{\alpha_{rs}}:\alpha_{rs}\in I(n-2,2n)\big)$. If $h\in\langle \Pi_{\alpha_{rs}}:\alpha_{rs}\in I(n-2,2n)\rangle$, we are done.
Otherwise, write $h=\sum_{\alpha\in I(n,2n)}A_{\alpha}X_{\alpha}$, with $A_{\overline{\alpha}}\neq 0$, where $\overline{\alpha}=(\overline{\alpha}_1,\ldots,\overline{\alpha}_n)\in I(n,2n)$.  The usual basis $e_{\alpha}$, for $\alpha\in I(n,2n)$, can be written as 
$${\mathcal B}=\{e_{(\beta,\overline{\alpha}_n)} : \beta\in I(n-1,2n-2),\; \overline{\alpha}_n\not\in\text{supp}( \beta)\}\cup\{e_{\alpha}:\alpha\in\Phi\}$$
where $\Phi=\{\beta\in I(n-1,2n-2):\overline{\alpha}_n\not\in\text{supp}(\beta)\}$. Then,
$$h=\sum_{(\beta,\overline{\alpha}_n)}A_{(\beta,\overline{\alpha}_n)}X_{(\beta,\overline{\alpha}_n)}+\sum_{\alpha\in\Phi}A_{\alpha}X_{\alpha}$$
with $\beta\in I(n-2,2n-2)$ y $\overline{\alpha}_n\not\in\text{supp}(\beta)$. Let
$$h'=\sum_{(\beta,\overline{\alpha}_n)}A_{(\beta,\overline{\alpha}_n)}X_{(\beta,\overline{\alpha}_n)}\quad\text{and}\quad h''=\sum_{\alpha\in\Phi}A_{\alpha}X_{\alpha},$$
where clearly $h',h''\in (\cuna^nE)^*$ and $h=h'+h''$.
 Let $\ell$ be the isotropic line generated by $e_{\overline{\alpha}_n}$ and let $U(\ell)=\{L\in L(n,2n):\ell\subseteq L\}$. We may identify $U(\ell)$ with the Lagrangian-Grassmannian $L(n-1,\ell^{\perp}/\ell)\simeq L(n-1,2n-2)$. Consider the map 
$$-\cunita e_{\overline{\alpha}_n}:\cuna^{n-1}E\rightarrow \cuna^nE$$
giving by wedging with $e_{\overline{\alpha}_n}$ on basis elements and then extending linearly. Composing this map with the contraction $f:\cuna^nE\rightarrow \cuna^{n-2}E$,  explicitely we have, for $\sum_{\beta\in (n-1,2n-2)}p_{\beta}e_{\beta}\in\cuna^{n-1}E$, 
\begin{align*}
\Big(\sum_{\beta\in I(n-1,2n-2)}p_{\beta}e_{\beta}\Big)\cunita e_{\overline{\alpha}_n}&=\sum_{\beta\in I(n-1,2n-2)}p_{(\beta,\overline{\alpha}_n)}(e_{\beta}\cunita e_{\overline{\alpha}_n})\\
&=\sum_{\beta\in I(n-1,2n-2)}p_{(\beta,\overline{\alpha}_n)}e_{(\beta,\overline{\alpha}_n)}=:w\in\cuna^nE,
\end{align*}
where the $\overline{\alpha}_n$ in the sums satisfy that $\overline{\alpha}_n\not\in \text{supp}(\beta)$. 
Appliying $f$ we obtain
\begin{align*}
f(w)&=\sum_{\beta\in I(n-1,2n-2)}p_{(\beta,\overline{\alpha}_n)}f(e_{(\beta,\overline{\alpha}_n)})\\
&=\sum_{\beta\in I(n-1,2n-2)}p_{(\beta,\overline{\alpha}_n)}\Big(\sum_{1\leq r<s\leq n}\langle e_{\alpha_r},e_{\alpha_s}\rangle e_{{(\beta,\overline{\alpha}_n)}_{rs}}\Big)\\
&= \sum_{1\leq r<s\leq n}\Big(\sum_{1\leq\varphi_1<\varphi_2\leq n}p_{(\beta,\overline{\alpha}_n)_{rs}\varphi_1\varphi_2}\langle e_{\varphi_1},e_{\varphi_2}\rangle\Big)e_{{(\beta,\overline{\alpha}_n)}_{rs}}.
\end{align*}

Now, if $w\in \text{Im}(-\cunita e_{\overline{\alpha}_n})\cap \ker f$, 
since the $e_{{(\beta,\overline{\alpha}_n)}_{rs}}$ are linearly independent, we must have that $\sum_{1\leq\varphi_1<\varphi_2\leq n}p_{((\beta,\overline{\alpha}_n)_{rs}\varphi_1\varphi_2)}\langle e_{\varphi_1},e_{\varphi_2}\rangle=0$, where $\langle e_{\varphi_1},e_{\varphi_2}\rangle=1$ if and only if $\varphi_1+\varphi_2=2n+1$, that is if $\varphi_1=i$, then $\varphi_2=2n-i+1$. Hence, the previous sum is $\sum_{i=1}^np_{(i,(\beta,\overline{\alpha}_n)_{rs},2n-i+1)}=0$ for all $\beta\in I(n-1,2n-2)$, again for $\overline{\alpha}_n\not\in\text{supp}(\beta)$, 
up to a permutation on the indexes. Therefore, $w$ satisfies the linear relations $\sum_{i=1}^nX_{(i,(\beta,\overline{\alpha}_n)_{rs},2n-i+1)}=0$, for all $\beta\in I(n-1,2n-2)$. 
Put $\Pi_{(\beta,\overline{\alpha}_n)_{rs}}:=\sum_{i=1}^nX_{(i,(\beta,\overline{\alpha}_n)_{rs},2n-i+1)}$. Clearly, $\{\Pi_{(\beta,\overline{\alpha}_n)_{rs}}:1\leq r<s\leq n,\beta\in I(n-1,2n-2)\}\subseteq \{\Pi_{\alpha_{rs}}:\alpha_{rs}\in I(n-2,2n)\}$,  
$$h'\big(L(n-1,\ell^{\perp}/\ell\big)({\mathbb F}_q)=h(U(\ell)({\mathbb F}_q))=0,$$
and by the induction hypothesis 
$$h'=\sum_{\beta\in I(n-1,2n-2)}A_{(\beta,\overline{\alpha}_n)}\Pi_{(\beta,\overline{\alpha}_n)}\in \langle \Pi_{\alpha_{rs}}:\alpha_{rs}\in I(n-2,2n)\rangle,$$
where $\overline{\alpha}_n\not\in\text{supp}(\beta)$ and 
$h'\neq 0$. Hence, appplying the same process to $h''$ we have that $h''=h'_1+h''_2$ with $h'_1\in\langle \Pi_{\alpha_{rs}}:\alpha_{rs}\in I(n-2,2n)\rangle$. This process must finish in a finite number of steps. 
\end{proof}
\begin{corollary}
$L(n,2n)({\mathbb F}_q)$ is an non-degenerate system in ${\mathbb P}(\ker f)({\mathbb F}_q)$ and the Lagrangian-Grassmannian $L(n,2n)$ is an \emph{FFN}$(1,q)$-projective variety in ${\mathbb P}(\ker f)$. 
\end{corollary}


We denote by
$C_{L(n,2n)}$  the $[n,k]_q$ nondegenerate linear code   induced by
the projective system $L(n,2n)({\mathbb F}_q)$. Here $n=\prod_{i=1}^n (1+q^i)$,   and $\; k=\binom
{2n}{n}-\rank B$, where $B$ is the matrix associated to the homogeneous
system of linear equations $\{\Pi_{\alpha_{rs}}: \alpha_{rs} \in
I(n-2, 2n) \}$.  A detailed description of $B$ and $\rank B$ is in \cite[Sections 3 and 4]{3.3}. 
For the minimum distance $d=d(L(n,2n)$ we have the bound $d< q^{\frac{n(n+1)}{2}}$, see \cite{3}. 
For some low dimension Lagrangian-Grassmannian codes their weight spectra have been completely determined, for example, for the Lagrangian-Grassmannian $C_{L(2,4)}$ code, by \cite{3.1} and \cite{3.2}, see also \cite{2.1}, and for the Lagrangian-Grassmannian $C_{L(3,6)}$ code in \cite{2.1}.
\end{example}

\begin{example}[Isotropic Grassmannians]\label{example3.4}
The set-up is the same as in Example \ref{example3.3}, that is, $E$ is a symplectic vector space of dimension $2n$ over a finite field ${\mathbb F}_q$ and $\overline{E}=E\otimes_{{\mathbb F}_q}\overline{\mathbb F}_q$. For any integer $1\leq \ell\leq n$, let $IG(\ell,2n)\subseteq G(\ell,2n)$ be the set of $k$-dimensional isotropic vector subspaces of $\overline{E}$. This is projective subvariety of ${\mathbb P}\big(\cuna^{\ell}\overline{E}\big)$, by means of the Pl\"ucker embedding.
The isotropic Grassmannian $IG(\ell,2n)$ is a  section
$$IG(\ell,2n)=G(\ell,2n)\cap L,$$
by a linear subspace $L\subseteq{\mathbb P}\big(\cuna^{\ell}\overline{E}\big)$ of codimension $\binom{2n}{\ell-2}$. Hence, its dimension is
$$\dim IG(\ell,2n)=\binom{2n}{\ell}-\binom{2n}{\ell-2}.$$
Also, the cardinality of its set of ${\mathbb F}_q$-rational points is
$$|IG(\ell,2n)({\mathbb F}_q)|=\prod_{i=0}^{\ell-1}\frac{q^{2n-2i}-1}{q^{i+1}-1}.$$
By Proposition \ref{prop2.3}, $IG(\ell,2n)({\mathbb F}_q)$ is a projective system in ${\mathbb P}(\cuna^{\ell}E)$. Thus, as in Definition \ref{def2.1}, it has an associated $[n,k]_q$-linear code $C_{IG(\ell,2n)}$ with parameters $n=|IG(\ell,2n)({\mathbb F}_q)|$ and $k=\binom{2n}{\ell}$. This family of codes was introduced and studied in  \cite{2.1}.
\end{example}
\begin{example}[Lagrangian-Schubert codes]\label{example3.5}
For the Schubert codes, in \cite{27}, \cite{11}, \cite{6} and \cite{29} their parameters are obtained by using a flag of vector spaces in the projective space ${\mathbb P}(\cuna^nE)$. Using similar ideas we look at a new code associated to a Schubert variety over a symplectic vector space. Again, let ${\mathbb F}_q$ be a finite field, $\overline{\mathbb F}_q$ an algebraic closure, and $E$ an ${\mathbb F}_q$-symplectic vector space of dimension $2n$. Let $L(n,2n)$ be the Lagrangian-Grassmannian variety defined in Example
\ref{example3.3}. Fix a flag of isotropic subspaces of $E$:
$$ 0\subset W_1\subset W_2\subset\cdots\subset W_n\subset E\leqno(*)$$
such that $\dim W_i=i$ for $1\leq i\leq n$. One such flag will be called an \textit{isotropic flag} of $E$. Notice that
since $W_n$ is  isotropic of dimension $n$, then $W_n\in L(n,2n)$. Thus, an isotropic flag of $E$ is just a complete flag of $W_n$. Observe that each isotropic flag of $E$ can be extended to a complete flag of $E$ by setting $W_{n+i}=W_{n-i}^{\perp}$, for $1\leq i\leq n$. Now, for a partition $\lambda=(\lambda_1,\ldots,\lambda_n)\in I(n,2n)$ and an isotropic flag $(\ast)$ of $E$, the \emph{Lagrangian-Schubert} variety is the set
$$L(n,2n)_{\lambda}:=\{W\in L(n,2n)\, :\, \dim(W\cap W_{n+1-\lambda_i})\geq i,\; 1\leq i\leq \ell(\lambda)\}$$
where $\ell(\lambda)=|\{p\in\{1,\ldots,n\}\, :\, \lambda_p\neq 0|$. $L(n,2n)_{\lambda}$ is
a subvariety of $L(n,2n)$ of codimension $|\lambda|:=\sum_{p=1}^{\ell}\lambda_p$ in ${\mathbb P}(\cuna^nE)$. Now, for
a partition $\lambda\in I(n,2n)$ and an isotropic flag $(\ast)$ of $E$, consider the subvarieties $L(n,2n)$ and
$\Omega_{\lambda}(n,2n)$. Then, $L(n,2n)_{\lambda}=L(n,2n)\cap \Omega_{\lambda}(n,2n)$. 
For the set of ${\mathbb F}_q$-rational points $L(n,2n)_{\lambda}({\mathbb F}_q)$, by \cite[Lemma 2.2]{3.1}, there exists an irreducible projective subvariety $Z\subseteq L(n,2n)_{\lambda}$ such that $L(n,2n)_{\lambda}({\mathbb F}_q)=Z({\mathbb F}_q)$  and $Z$ is an FFN$(1,q)$-variety. Therefore, by \cite[Corollaries 3 and 4]{3.1},  $L(n,2n)_{\lambda}({\mathbb F}_q)$ defines a linear code $C_{L(n,2n)_{\lambda}}$ whose parameters are given as follows:
let $r=r(\lambda)=|L(n,2n)_{\lambda}|$ be the number of rational points of $L(n,2n)_{\lambda}$ over the finite field ${\mathbb F}_q$. Write $L(n,2n)_{\lambda}({\mathbb F}_q)=\{P_1,\ldots,P_r\}$ with the $P_i$ representatives of the corresponding points in ${\mathbb P}(\cuna^nE)$, under the Pl\"ucker embedding. Let $K=\{h\in (\cuna^nE)^*\, :\, h(P_1)=\cdots=h(P_r)=0\}$ and $V=\{w\in\cuna^nE\, :\, h(w)=0\;\;\text{for all}\;\; h\in K\}$ as in the beginning of this Section. Then, by \cite[Section 3.2]{ikeda}, the lenght of the code $C_{L(n,2n)_{\lambda}}$ is
 $\big|L(n,2n)_{\lambda}({\mathbb
F}_q)\big|=\sum_{\beta\not\leq\gamma}q^{\delta_{\beta}}$, where $\delta_{\beta}$ is the dimension of the affine space isomorphic to the corresponding Schubert cell. 
The dimension of $C_{L(n,2n)_{\lambda}}$ is  $k=\binom{2n}{n}-\dim B$, where $B$ is the matrix associated to the system of linear equations $Z(\Pi_{\alpha_{rs}},X_{\beta}: \alpha_{rs}\in I(n-2,2n), \beta\not\leq\gamma)$.
The minimum distance $d$ of $C_{L(n,2n)_{\lambda}}$ satisfies the bound $d\leq q^{\dim L(n,2n)_{\lambda}}$, by  \cite[Corollary 4]{3.1}.
\end{example}

\begin{example}[Lagrangian-Schubert union codes]
For $\lambda_i\in I(n,2n)$, $1\leq i\leq r$, 
let $L_U=\bigcup_{i=1}^r L(n,2n)_{\lambda_i}$ and $S_U=\bigcup_{i=1}^r\Omega_{\lambda_i}$. Let $H_U=I(n,2n)-G_U$, for $G_U$ as in Example \ref{example3.2}. Then,
\begin{align*}
L_U&=G(n,2n)\cap Z(\Pi_{\alpha_{rs}}, X_{\beta}: \alpha_{rs}\in I(n-2,2n), \beta\in H_U)=S_U\cap{\mathbb P}(\ker f)\\
&=L(n,2n)\cap Z(X_{\beta}:\beta\in H_U).
\end{align*}
Now, as in Example \ref{example3.5}, for the set of ${\mathbb F}_q$-rational points $L_U({\mathbb F}_q)$, there exists an irreducible FFN$(1,q)$-subvariety $Y\subseteq L_U$ such that $Y({\mathbb F}_q)=L_U({\mathbb F}_q)$ and a corresponding linear code $C_{L_U}$ whose parameters satisfy:
 $n=\big|L_U({\mathbb F}_q)\big|=\sum_{\beta\in H_U}q^{\delta_{\beta}}$, 
$k=\binom{2n}{n}-\dim B$, where $B$ is the matrix associated to the system of linear equations $Z(\Pi_{\alpha_{rs}},X_{\beta}: \alpha_{rs}\in I(n-2,2n), \beta\in H_U)$.
The minimum distance $d$ of $C_{L_U}$ satisfies the bound $d\leq q^{\dim L_U}$.
\end{example}

\section{Higher weights of the Lagrangian-Grassmannian codes}\label{sec4}

Now, we address the question of finding bounds for the higher weights of the La\-gran\-gian-Grass\-mannian code and to do this we use what is known, \cite{5}, \cite{6} and \cite{17},  for the higher weights $d_r(C(n,2n))$ of the linear code $C(n,2n)$ associated to the ${\mathbb F}_q$-rational points $G(n,2n)({\mathbb F}_q)$ of the Grassmannian. By
\cite[Thm. 4]{6}, 
\begin{align*}
d_r(C(n,2n))&\geq q^{\delta}+\cdots+q^{\delta-r+1},\; \text{where $\delta=n(2n-n)=n^2$},\\
d_r(C(n,2n))&=q^{\delta}+\cdots+q^{\delta-r+1},\;  \text{if $1\leq r\leq \max\{n,2n-n+1\}=n+1$}.
\end{align*}
Now for the higher weights $d_r(C_{L(n,2n)})$ of the Lagrangian-Grassmannian code, with the notation of Example \ref{example3.3}, suppose $H$ is a codimension $r$ linear subvariety of ${\mathbb P}(V)$. Then,  $H$ has codimension $r'$ in ${\mathbb P}(\wedge^nE)$, for $r'>r$. Observe now that
$$|L(n,2n)({\mathbb F}_q)\cap H|=|G(n,2n)\cap {\mathbb P}(V)\cap H|=|G(n,2n)({\mathbb F}_q)\cap H|$$
and thus, for $H\subseteq {\mathbb P}(V)$:
\begin{align*}
\max_{\codim H=r}\{|L(n,2n)({\mathbb F}_q)\cap H|\}
&= \max_{H\subseteq {\mathbb P}(V)}\{|G(n,2n)({\mathbb F}_q)\cap H| : \codim H=r\}\\
&\leq \max_{H\subseteq {\mathbb P}(\wedge^nE)}\{|G(n,2n)({\mathbb F}_q)\cap H| : \codim H=r'\}.
\end{align*}
Therefore,
$$
|L(n,2n)({\mathbb F}_q)|-\max_{H\subseteq {\mathbb P}(\wedge^nE)}\{|G(n,2n)({\mathbb F}_q)\cap H| :\codim H=r'\}\qquad\qquad\qquad\qquad\qquad\qquad$$
$$\qquad\qquad\qquad\qquad \leq |L(n,2n)({\mathbb F}_q)|-\max_{H\subseteq {\mathbb P}(V)}\{|L(n,2n)({\mathbb F}_q)\cap H| : \codim H=r\}.
$$
Hence,
$$
|L(n,2n)({\mathbb F}_q)|-|G(n,2n)({\mathbb F}_q)|+\Big( |G(n,2n)({\mathbb F}_q)|-\max_{H\subseteq {\mathbb P}(\wedge^nE)}\{|G(n,2n)({\mathbb F}_q)\cap H|\}\Big)
$$
$$\qquad\qquad\qquad\qquad \qquad\qquad\qquad\qquad\leq d_r(C_{L(n,2n)}),$$
where the maximum is taken over all linear subvarieties $H\subseteq {\mathbb P}(\wedge^nE)$ such that $\codim H=r'$. We have proved: 
\begin{proposition}\label{prop11}\label{proposition13}
With the notation above,
\begin{equation*}
d_r(C_{L(n,2n)})\geq |(L(n,2n)({\mathbb F}_q)|-|G(n,2n)({\mathbb F}_q)|+d_{r'}(C(n,2n))\end{equation*}
and
\begin{align*}
|L(n,2n)({\mathbb F}_q)|-|G(n,2n)({\mathbb F}_q)|+d_{r'}(C(n,2n))
&\leq d_r(C_{L(n,2n)})\\
&\leq |L(n,2n)({\mathbb F}_q)|-\dim V+r, 
\end{align*}
where $r'=\binom{2n}{n}-\dim V+r$ and
$$ |G(n,2n)({\mathbb F}_q)|=
\displaystyle\begin{bmatrix}
2n\\
n
\end{bmatrix}_q
=\frac{(q^{2n}-1)(q^{2n}-q)\cdots (q^{2n}-q^{n-1})}{(q^n-1)(q^n-q)\cdots (q^n-q^{n-1})}.$$
\end{proposition}

We just note that the second bound is obtained from the generalized Singleton bound \cite{25}.

Now, following   \cite[Section 5]{5}, fix a set $T(n,2n)=\{w_1,\ldots, w_t\}$ of representatives in $\cuna^nE$ corresponding to points in $G(n,2n)({\mathbb F}_q)$. 
Given a subspace $S$ of $\cuna^nE$, we put $g(S)=|S\cap T(n,2n)|$ and let
$$g_r(n,2n)=\max\{g(S)\; :\; \text{$S$ is a codimension $r$ subspace of $\textstyle\cuna^nE$}\}.$$

For $f:\cuna^nE\rightarrow \cuna^{n-2}E$ as in Section \ref{sec2}, let $V=\ker f$. Then, $|L(n,2n)({\mathbb F}_q)|=g(V)$ and it is immediate that $g(V)\leq g_r(n,2n)$ for $r=\binom{2n}{n}-\dim V$. Therefore, from  \cite[Corollary 17]{5} we obtain
\begin{align*}
d_r(C(n,2n))&=|G(n,2n)({\mathbb F}_q)|-g_r(n,2n)\\
&\leq |G(n,2n)({\mathbb F}_q)|-g(V)=|G(n,2n)({\mathbb F}_q)|-|L(n,2n)({\mathbb F}_q)|.
\end{align*}
We have proved:

\begin{proposition}\label{proposition14}
If $1\leq r\leq \binom{2n}{n}$, then 
$$d_r(C(n,2n))\leq |G(n,2n)({\mathbb F}_q)|-|L(n,2n)({\mathbb F}_q)|.$$
\end{proposition}

Again, following \cite{5}, we address the problem of determining the maximum number of points on linear sections $H\cap L(n,2n)({\mathbb F}_q)$ of the Lagrangian-Grassmannian, for linear subvarieties $H\subseteq {\mathbb P}(\cuna^nE)$ of codimension $r$. This problem can be translated to the same problem for the Grassmann variety as follows: Let $H$ be a codimension $r$ linear subvariety of the projective space ${\mathbb P}(V)=Z\langle g_1,\ldots,g_N\rangle$. We want to calculate the number of points in the intersection $H\cap L(n,2n)({\mathbb F}_q)$. Now, since $H$ is a codimension $r$ linear subvariety of ${\mathbb P}(V)$, then $H=Z\langle g_1,\ldots,g_N,h_1,\ldots,h_t\rangle$ is a linear subvariety of ${\mathbb P}(\cuna^nE)$ of codimension $r'>r$, where $\{h_1,\ldots, h_t\}$ is a set of linear homogeneous polynomials. Then, the problem we are addressing could be translated to the problem of finding the number of points of the intersection of Grassmannian $G(n,2n)({\mathbb F}_q)$ and the codimension $r'$ linear subvariety of ${\mathbb P}(\cuna^nE)$. If the linear subvariety of ${\mathbb P}(V)$ is of the form $H_{\Lambda}=Z\langle g_1,\ldots,g_N,x_{\alpha}\, : \, \text{for all $\alpha\in \Lambda$}\rangle$, where $\Lambda\subseteq I(n,2n)$ is a \textit{close} family (see \cite{5}) with $k$ elements, we have an upper bound for the number of points in the intersection $H_{\Lambda}\cap L(n,2n)({\mathbb F}_q)$. Indeed, if $H_{\Lambda}'=Z\langle x_{\alpha}\, :\, \text{for all $\alpha\in \Lambda$}\rangle$ is the corresponding linear subvariety of ${\mathbb P}(\cuna^nE)$, then
$$H_{\Lambda}\cap L(n,2n)({\mathbb F}_q)=H_{\Lambda}'\cap G(n,2n)({\mathbb F}_q)\cap {\mathbb P}(V)\subseteq H_{\Lambda}'\cap G(n,2n)({\mathbb F}_q)$$
and using  \cite[3.2 ]{5}, we have
$$|H_{\Lambda}'\cap G(n,2n)({\mathbb F}_q)|=\begin{bmatrix}
2n\\
n
\end{bmatrix}_q-q^{n^2}-q^{n^2-1}-\cdots -q^{n^2-k+1}.$$
We have proved that:

\begin{proposition}\label{proposition15}
If $\Lambda$ is a close family of $I(n,2n)$ with $k$ elements, then
$$|H_{\Lambda}\cap L(n,2n)({\mathbb F}_q)|\leq \begin{bmatrix}
2n\\
n
\end{bmatrix}_q-q^{n^2}-q^{n^2-1}-\cdots -q^{n^2-k+1}.$$
\end{proposition}


\bibliographystyle{model1-num-names}

\end{document}